\documentclass[a4paper,11pt]{article}
\usepackage{amsmath}
\usepackage{amsthm}
\usepackage{amssymb}
\usepackage{mathtools}
\usepackage{mathrsfs}
\usepackage{stmaryrd}
\usepackage{physics}
\usepackage{xspace}
\usepackage[margin=1.0in]{geometry}
\usepackage[ruled,vlined]{algorithm2e}
\usepackage[hidelinks]{hyperref}

\newcommand\N{\mathbb N}

\newcommand\R{\mathbb R}

\newcommand\PP{\mathbb P}
\newcommand\E{\mathbb E}
\newcommand\T{\mathbb T}
\newcommand\F{\mathscr F}
\newcommand\G{\mathscr G}
\newcommand\X{\mathbf X}
\newcommand\Y{\mathbf Y}
\newcommand\Ind{\boldsymbol 1}

\DeclarePairedDelimiter{\floor}{\lfloor}{\rfloor}

\DeclarePairedDelimiterX{\infdivx}[2]{(}{)}{%
  #1\;\delimsize|\delimsize|\;#2%
}
\newcommand{\kld}[2]{\ensuremath{\mathbf{d}\infdivx{#1}{#2}}\xspace}

\newcommand{\EX}[1]{\E\left[#1\right]}
\newcommand{\cEX}[2]{\E\left[#1 \,\middle|\, #2 \right]}

\usepackage{xcolor}

\theoremstyle{plain}
\newtheorem{theorem}{Theorem}[section]
\newtheorem*{theorem*}{Theorem}
\newtheorem{lemma}[theorem]{Lemma}
\newtheorem*{lemma*}{Lemma}
\newtheorem{corollary}[theorem]{Corollary}
\newtheorem*{corollary*}{Corollary}
\newtheorem{proposition}[theorem]{Proposition}
\newtheorem*{proposition*}{Proposition}
\newtheorem{fact}[theorem]{Fact}
\newtheorem*{fact*}{Fact}

\theoremstyle{definition}
\newtheorem{definition}[theorem]{Definition}
\newtheorem{definition*}[theorem]{Definition}
\newtheorem{assumption}[theorem]{Assumption}

\theoremstyle{remark}
\newtheorem{remark}[theorem]{Remark}
\newtheorem*{remark*}{Remark}

\numberwithin{equation}{section}

\begin{document}

\title{Efficient approximation of branching random walk Gibbs measures}
\author{
Fu-Hsuan~Ho\thanks{
  Institut de Math\'{e}matiques de Toulouse, CNRS UMR5219.
  \textit{Postal address:} Institut de Math\'{e}matiques de Toulouse,
  Universit\'{e} Toulouse 3 Paul Sabatier,
  118 Route de Narbonne, 31062 Toulouse Cedex 9, France.
  \textit{Email:} \texttt{fu-hsuan.ho AT math.univ-toulouse.fr}, \texttt{pascal.maillard AT math.univ-toulouse.fr}. Supported in part by grants ANR-20-CE92-0010-01 and ANR-11-LABX-0040 (ANR program ``Investissements d'Avenir'').
  }
\and
Pascal~Maillard\footnotemark[1]
}
\date{\today}
\maketitle

\begin{abstract}
Disordered systems such as spin glasses have been used extensively as models for high-dimensional random landscapes and studied from the perspective of optimization algorithms. In a recent paper by L. Addario-Berry and the second author, the \emph{continuous random energy model (CREM)} was proposed as a simple toy model to study the efficiency of such algorithms. The following question was raised in that paper: what is the threshold $\beta_G$, at which sampling (approximately) from the Gibbs measure at inverse temperature $\beta$ becomes algorithmically hard?

This paper is a first step towards answering this question. We consider the branching random walk, a time-homogeneous version of the continuous random energy model. We show that a simple greedy search on a renormalized tree yields a linear-time algorithm which approximately samples from the Gibbs measure, for every  $\beta < \beta_c$, the (static) critical point. 
More precisely, we show that for every $\varepsilon>0$, there exists such an algorithm such that the specific relative entropy between the law sampled by the algorithm and the Gibbs measure of inverse temperature $\beta$ is less than $\varepsilon$ with high probability. 

In the supercritical regime $\beta > \beta_c$, we provide the following hardness result. Under a mild regularity condition, for every $\delta > 0$, there exists $z>0$ such that the running time of any given algorithm approximating the Gibbs measure stochastically dominates a geometric random variable with parameter $e^{-z\sqrt{N}}$ on an event with probability at least $1-\delta$.

\paragraph{Keywords:} branching random walk ; Gibbs measure ; Kullback--Leibler divergence.
 
\paragraph{MSC2020 subject classifications:} 68Q17, 82D30, 60K35, 60J80. 
\end{abstract}

\section{Introduction} \label{sec:intro}


We consider the following family of branching random walks. An initial particle is located at the origin. It gives birth to $d$ child particles, $d\geq 2$, scattering on the real line, and each of the child particles produces $d$ child particles again, and so on. The displacement of each particle is independent of the past of the process and of the displacements of its siblings. The genealogy of the particles can be represented by a $d$-ary tree $\T_N$, identifying particles with the vertices of the tree. We denote by $X_v$ the location of a particle $v\in\T_N$.

Addario-Berry and Maillard \cite{AB&M20} considered algorithms that, for a given $x>0$, find a leaf $v$ of $\T_N$ such that $X_v\geq xN$ in the framework of the continuous random energy model (CREM). This is a binary time-inhomogeneous branching random walk with Gaussian displacements. More precisely, the CREM is a Gaussian process whose covariance function is given by 
\[\EX{X_vX_w} = A\left(\frac{\abs{v\wedge w}}{N}\right), \quad \forall v,w\in\T_n\]
where $A:[0,1]\rightarrow [0,1]$ is an increasing function with $A(0)=0$ and $A(1)=1$ and $\abs{v\wedge w}$ is the depth of the most recent common ancestor of $v$ and $w$.
The authors proved the existence of a threshold $x_*$ such that the following holds: a) for every $x<x_*$, there exists a polynomial-time algorithm that can accomplish the task with high probability, b) for every $x>x_*$, every such algorithm has a running time which is at least exponential in $N$ with high probability. The authors also raised the question of the complexity of sampling a \emph{typical} vertex of value roughly $xN$, which can be interpreted as sampling a vertex according to a Gibbs measure with a certain parameter $\beta$ depending on $x$.  

The present work attacks this problem in the simpler setting of the (homogeneous) branching random walk, corresponding to the case $A(x)=x$ of the CREM in the special case of Gaussian displacements. The Gibbs measure is a probability measure on the leaves $v$ of $\T_N$ with weight proportional to $e^{\beta X_v}$, where $\beta>0$ is a given parameter called the \emph{inverse temperature}. We show that there exists a threshold $\beta_c>0$ such that the following holds: a) in the subcritical regime $\beta<\beta_c$, there exists a linear-time algorithm such that with high probability, the specific relative entropy between the law sampled by the algorithm and the Gibbs measure of inverse temperature $\beta$ is arbitrarily small, b) in the supercritical regime $\beta>\beta_c$, under a mild regularity condition, we show that with high probability, the running time of any given algorithm approximating the Gibbs measure in this sense is at least stretched exponential in $N$.

\subsection{Notations and main results} \label{sec:notation}

Let $\T$ be a rooted $d$-ary tree, where $d\geq 2$. The depth of a vertex $v\in \T$ is denoted by $\abs{v}$. We denote the root by $\varnothing$, and
any vertex $v$ with depth $n\geq 1$ is indexed by a string $v_1\cdots v_n\in\{1,\ldots,d\}^n$. For any $v,w\in\T$, we write $v\leq w$ if $v$ is a prefix of $w$ and write $v<w$ if $v$ is a prefix of $w$ strictly shorter than $w$.
We denote $\T_n$ to be the subtree of $\T$ containing vertices of depth less or equal to $n$
and $\partial\T_n$ to be the leaves of $\T_n$. 

Let $\Y = (Y_0,\ldots,Y_{d-1})$ be a $d$-dimensional random vector. Let $(\Y^v)_{v\in\T}$ be iid copies of $\Y$ where $\Y^v = (Y_{v0},\ldots, Y_{v(d-1)})$ -- this uniquely defines $Y_u$ for every $u\in\T\backslash\{\varnothing\}$. Define the process $\X = (X_v)_{v\in\T}$ by
\begin{align*}
\begin{cases}
X_\varnothing = 0, & \\[6pt]
X_v = \sum_{\varnothing < w \le v} Y_w, & \abs{v} \ge 1.
\end{cases}
\end{align*}
The process $\X$ is called the \emph{branching random
walk with increments $\Y$}.

It is well-known that $\X$ has the branching property: let $\F=(\F_n)_{n\geq 0}$ be its natural filtration. For any $v\in\T$ with $\abs{v} = n$, define
\[\X^v = (X^v_w)_{\abs{w}\geq 0} = (X_{vw}-X_v)_{\abs{w}\geq 0}.\]
Then $(\X^v)_{|v|=n}$ are iid copies of $\X$ and independent of $\F_n$.

\paragraph{Gibbs measures.}
Define the following function
\begin{align*}
  \varphi(\beta) = \log\left(\EX{\sum_{i=0}^{d-1} e^{\beta Y_i}}\right)\in (-\infty,\infty],\quad \beta\in\R,
\end{align*}
and set $\mathcal D(\varphi) = \{\beta\in\R: \varphi(\beta)<\infty\}$.
It is well known that $\varphi$ is convex and that it is smooth on $\mathcal D(\varphi)^\circ$, the interior of $\mathcal D(\varphi)$. Throughout the article, we assume
the following.
\begin{assumption} \label{assump:phi}
$0\in\mathcal D(\varphi)^\circ$.
\end{assumption}

For $n\in\N$, define the following (normalized) partition functions
\begin{align*}
  W_{\beta,n} &= \sum_{\abs{v} = n} e^{\beta X_v - \varphi(\beta)n}, \quad n\geq 0, \\
  W^u_{\beta,n} &= \sum_{\abs{w}=n} e^{\beta X_w^u - \varphi(\beta)n},\quad u\in \T \text{ and } n\geq 0.
\end{align*}
Note that for every $m\le n$,
\begin{align} \label{eq:W decomp}
  W_{\beta,n} = \sum_{\abs{u}=m} e^{\beta X_u - \varphi(\beta)m}\cdot W^u_{\beta,n-m}.
\end{align}
For $\beta\in\R$ and $n\in\N$, we now define the Gibbs measure of parameter $\beta$ on $\T_n$ to be
\begin{align} \label{def:Gibbs0}
  \mu_{\beta,n}(u) = e^{\beta X_u - \varphi(\beta)m}\cdot \frac{W^u_{\beta,n-m}}{W_{\beta,n}},\quad \abs{u}=m\leq n.
\end{align}
Note that $\mu_{\beta,n}$ is usually defined on $\partial \T_n$ only, but it will be helpful to define it on the whole tree $\T_n$. By \eqref{eq:W decomp}, for every $m\le n$, the restriction of $\mu_{\beta,n}$ to $\partial \T_m$ is a probability measure. Similarly, we can define
\begin{align}
\mu_{\beta,n}^v(u)=e^{\beta X^v_u-\varphi(\beta)m}\cdot\frac{W^{vu}_{\beta,n-m}}{W^v_{\beta,n}},\quad v\in\T \text{ and } \abs{u}=m\leq n. 
\label{def:Gibbs1}
\end{align}

The free energy of the branching random walk has been calculated by Derrida and Spohn~\cite{DerridaSpohn} (and can also be deduced from Biggins~\cite{BigginsChernoff})
\[
\lim_{n\to\infty} \frac 1{\beta n} \log \sum_{\abs{v} = n} e^{\beta X_v} =  \begin{cases}
\frac 1 \beta \varphi(\beta)&\text{if $\beta \in (0,\beta_c)$}\\
\frac 1 {\beta_c} \varphi(\beta_c) &\text{if $\beta \ge \beta_c$},
\end{cases}
\]
where the limit is meant to be in probability and where the critical inverse temperature $\beta_c$ is defined by 
\[\beta_c=\sup\{\beta \in \mathcal D(\varphi)^\circ \mid \beta\varphi'(\beta) < \varphi(\beta)\} \in (0,\infty].\]
We will mostly be interested in the phase $\beta < \beta_c$. In this phase, we recall the following result.
\begin{fact}[\cite{KP76,Biggins77,Lyons97}]
\label{fact:W}
If $\beta < \beta_c$, then the martingale $(W_{\beta,n})_{n\ge0}$ is uniformly integrable. In fact, there exists a (strictly) positive random variable $W_{\beta,\infty}$ with $\E[W_{\beta,\infty}] = 1$ and such that $W_{\beta,n}\to W_{\beta,\infty}$ almost surely and in $L^1$ as $n\to\infty$.
\end{fact}

\paragraph{Algorithms.}
We define an algorithmic model similar to the ones in \cite{PemantleSearch09,AB&M20}.
Let $N\in\N$. A random sequence $\mathrm{v}=(v(k))_{k\geq 0}$ taking values in $\mathbb{T}_N$ is called a \emph{(randomized) algorithm} if $v(0) = \varnothing$ and $v(k+1)$ is $\tilde{\F}_k$-measurable for every $k\geq 0$. Here, 
\[
\tilde{\F}_k = \sigma\left(v(1),\ldots,v(k);\,X_{v(1)},\ldots,X_{v(k)};\,U_1,\ldots,U_{k+1}\right)
\]
where $(U_k)_{k\geq 1}$ is a sequence of iid uniform random variables on $[0,1]$, independent of the branching random walk $\mathbf{X}$. Roughly speaking, the filtration $\tilde{\F} = (\tilde{\F}_k)_{k\geq 0}$ contains all information about everything we have queried so far, as well as the additional randomness needed to choose the next vertex. We further suppose that there exists a stopping time $\tau$ with respect to the filtration $\tilde{\F}$ and such that $v(\tau) \in \partial \T_N$. We call $\tau$ the \emph{running time} and $v(\tau)$ the \emph{output} of the algorithm. The \emph{law of the output} is the (random) distribution of $v(\tau)$, conditioned on the branching random walk.

Often, we will consider a family of algorithms indexed by $N$, which we also call an algorithm by abuse of notation.
We say that the algorithm is a $\emph{polynomial-time}$ algorithm if there exists a (deterministic) polynomial $P(N)$ such that $\tau \le P(N)$ almost surely.

Fix $M\in\N$. Given a configuration of the branching random walk of depth $N$,
consider the following algorithm (See Figure \ref{fig:algor}):

\begin{center}
  \begin{algorithm}[H]
    \label{algor}
    set $v=\varnothing$\;
    \While{$\abs{v}<N$} {
      choose $w$ with $\abs{w}=M\wedge (N-\abs{v})$ according to the Gibbs measure $\mu^v_{\beta,M\wedge (N-\abs{v})}$\;
      replace $v$ with $vw$\;
    }
    output $v$
    \caption{Recursive sampling on $M$-renormalized tree}
  \end{algorithm}
\end{center}

\begin{figure}[ht]
  \label{fig:algor}
  \centering
  \includegraphics[width=0.8\textwidth]{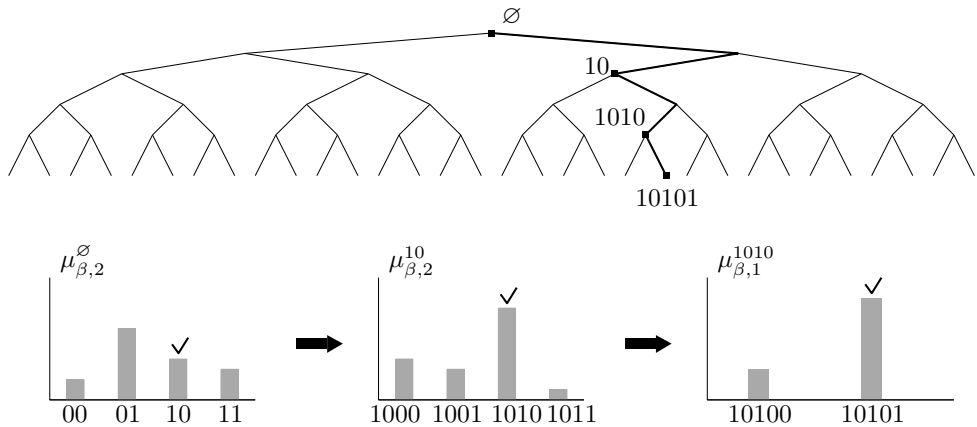}
  \caption{Schematic illustration of Algorithm \ref{algor} with $N=5$ and $M=2$ on a given configuration of the binary branching random walk. The squares are the vertices that the algorithm chooses at each step, and the thick line is the path from the root to the sampled vertex. The bar charts represent the Gibbs measures from which the algorithm samples at each step, and the check mark indicates which vertex is chosen. Note that the histograms are not drawn to scale.}
\end{figure}

\begin{remark}
\label{rem:algo}
It is easy to see that Algorithm~\ref{algor} can be formally written as a randomized algorithm according to the above algorithmic model. Furthermore, its running time is deterministic and bounded by $\lceil N/M \rceil 2^M$. The law of its output is a random probability measure $\mu_{\beta,M,N}$ on $\partial \T_N$ that can be recursively defined as follows:

  \begin{equation}
  \begin{split}
    \mu_{\beta,M,0}(\varnothing) &= 1 \\
    \mu_{\beta,M,N\wedge (K+1)M}(vw) &= \mu_{\beta,M,KM}(v)\cdot \mu^v_{\beta,M\wedge(N-KM)}(w) 
  \end{split}
  \label{eq:mu_decomposition}
  \end{equation}
  for all $\abs{v}=KM$, $\abs{w}=M\wedge(N-KM)$ and $0\leq K\leq \floor*{\frac{N}{M}}$.
\end{remark}

\paragraph{Approximation and threshold.}
Given two probability measures $P$ and $Q$ defined on a discrete space $\Omega$, the entropy of $Q$ and the Kullback--Leibler divergence from $Q$ to $P$ are respectively defined by 
\begin{align} 
\label{def:H}
H(Q) &= \sum_{\omega\in \Omega} Q(\omega)\cdot \log\left(\frac 1 {Q(\omega)}\right)\\
\label{def:KL}
\kld{P}{Q} &= \sum_{\omega\in \Omega} P(\omega)\cdot\log\left(\frac{P(\omega)}{Q(\omega)}\right).
\end{align}

Note that by Jensen's inequality, the entropy and the Kullback--Leibler divergence are non-negative. In what follows, we will often take $P$ and/or $Q$ to be a Gibbs measure $\mu_{\beta,N}$ for some $N$. In that case, we set $\Omega=\partial \T_n$, where $n$ is the largest number such that both $P$ and $Q$ are defined on $\partial \T_n$.

The following lemma is folklore. For completeness, we provide a proof in Section~\ref{sec:proof_lemma}.
\begin{lemma}
\label{lem:entropy}
\ 
\begin{enumerate}
\item If $\beta \in[0,\beta_c)$, then $H(\mu_{\beta,N})/N$ converges in probability to a positive constant as $N\to\infty$.
\item If $\beta > \beta_c$ and if $\beta_c\in\mathcal D(\varphi)^\circ$, then $H(\mu_{\beta,N}) = O(1)$ in probability, as $N\to\infty$. In other words, the sequence of random variables $(H(\mu_{\beta,N}))_{N\ge1}$ is tight.
\end{enumerate}
\end{lemma}

We are interested in the  \emph{algorithmically efficient approximation} of the Gibbs measure $\mu_{\beta,N}$. The notion of approximation we will use is the following.
\begin{definition}
\label{def:approximation}
Let $\beta \ge0$. We say that a sequence of random probability measures $(\tilde \mu_{\beta,N})_{N\ge1}$ \emph{approximates} the Gibbs measure $\mu_{\beta,N}$ if
\begin{equation}
\frac{\kld{\tilde \mu_{\beta,N}}{\mu_{\beta,N}}}{H(\mu_{\beta,N})} \to 0\;\; \text{in probability as $N\to\infty$}.
\end{equation}
\end{definition}

\begin{remark}
  More generally, assume that we are given two sequences of probability measures $(P_N)_{N\geq 1}$ and $(Q_N)_{N\geq 1}$ satisfying $H(Q_N)/N \to C\in(0,\infty)$ and 
\begin{align}
\label{eq:measure equiv}
\frac{1}{N}\kld{P_N}{Q_N}\rightarrow 0 \quad \text{as} \quad N\rightarrow\infty.
\end{align}
This has been called \emph{measure equivalence} or \emph{equivalence in the sense of specific relative entropy} in the physics literature \cite{TouchetteEquivalence}. 
Mathematically, Equation~\eqref{eq:measure equiv} implies the following: if $(A_N)_{N\geq 1}$ is a sequence of sets such that $Q_N(A_N)$ convergences to $0$ exponentially fast as $N\rightarrow\infty$, then we also have $P_N(A_N)\to 0$. Indeed, this is an easy consequence of Birg\'{e}'s inequality (see e.g. Theorem~4.20~in~\cite{ConIneq13}).

\end{remark}

\paragraph{Main results.}

We now state the main theorem of this paper.

\begin{theorem}[Approximation bounds]
  \label{thm:main}Let $N\in\N$, $M\in\llbracket 1,N \rrbracket$, and $\beta\in [0,\beta_c)$. Then for all $p\geq1$, there exists a constant $C_1(p) > 0$ such that
  \begin{align}
  \norm{\kld{\mu_{\beta,M,N}}{\mu_{\beta,N}}}_p \leq C_1(p)\cdot \floor*{\frac{N}{M}}.\label{eq:main.1}
  \end{align}
  Moreover, for all $p\geq1$, there exists a constant $C_1(p)>0$ such that 
  \begin{align}
  \norm{\kld{\mu_{\beta,M,N}}{\mu_{\beta,N}}-\EX{\kld{\mu_{\beta,M,N}}{\mu_{\beta,N}}}}_p \leq C_1(p). \label{eq:main.2}
  \end{align}
\end{theorem}

By Theorem \ref{thm:main}, we can derive the following corollary which states the existence of an algorithm that approximates the Gibbs measure $\mu_{\beta,N}$ efficiently.

\begin{corollary}[Complexity upper bound] \label{cor:main}
If $\beta\in [0,\beta_c)$, then 
there exists a polynomial-time algorithm such that for every $p>1$, denoting by $\tilde{\mu}_{\beta,N}$ the law of its output,
\begin{align}
\label{eq:corollary}
  \frac{1}{N}\norm{\kld{\tilde{\mu}_{\beta,N}}{\mu_{\beta,N}}}_p \rightarrow 0
\end{align}
as $N\rightarrow\infty$. In particular, $\tilde \mu_{\beta,N}$ approximates the Gibbs measure $\mu_{\beta,N}$ in the sense of Definition~\ref{def:approximation}.
\end{corollary}
\begin{proof}
Let $M= M(N)$ be a sequence that goes to infinity as $N\rightarrow\infty$, and set $\tilde{\mu}_{\beta,N} = \mu_{\beta,M,N}$. Equation \eqref{eq:corollary} then follows from \eqref{eq:main.1} in Theorem~\ref{thm:main}. Assuming moreover that $M = O(\log N)$, Remark~\ref{rem:algo} implies  that $\tilde{\mu}_{\beta,N}$ can be computed by a polynomial-time algorithm. The second statement follows from the first part of Lemma~\ref{lem:entropy}.
\end{proof}

Finally, we also provide a hardness result, assuming a mild regularity condition.

\begin{theorem}[Complexity lower bound]
\label{th:hardness}
Assume $\beta_c\in \mathcal D(\varphi)^\circ$ (in particular, $\beta_c < \infty$).
Let $\beta > \beta_c$. Let $v = (v(k))_{k\ge0}$ be an algorithm which outputs a vertex of law $\tilde \mu_N$ such that $\tilde \mu_N$ approximates the Gibbs measure $\mu_{\beta,N}$ in the sense of Definition~\ref{def:approximation}. Let $\tau$ be the running time of the algorithm. Then for every $\delta>0$, there exists $z>0$, such that for large enough $N$,
\[
\mathbb P\left(\tau \ge e^{z\sqrt N}\right) \ge 1-\delta.
\]
\end{theorem}

\subsection{Related work} \label{sec:related work}

An early study on searching algorithms on the branching random walk can be found in Karp and Pearl~\cite{KP83}. They considered the binary branching random walk with Bernoulli increments $\text{Ber}(p)$, and they showed that for $p>1/2$ and $p=1/2$, they gave an algorithm that can find an exact maximal vertex in linear and quadratic expected time, respectively. While for $p<1/2$, 
it is possible to find an approximate maximal vertex in linear time with high probability using a depth-first search on a renormalized tree. Aldous~\cite{Aldous92} gave a different algorithm and, among other things, extended the result of Karp and Pearl~\cite{KP83} to general increment distributions. A hardness result was obtained by Pemantle~\cite{PemantleSearch09}. Among other things, he showed for the binary branching random walk with Bernoulli increments with mean $p < 1/2$ that any search algorithm which finds a vertex within a $(1-\varepsilon)$ factor of the maximum with high probability needs at least $N \times \exp(\Theta(\varepsilon^{-1/2}))$ with high probability.\footnote{To be precise, this explicit bound relied on a conjecture on branching random walk killed at a linear space-time barrier (Conjecture 1 in \cite{PemantleSearch09}), which was subsequently proven to be true \cite{BG11,GHS11}.}


As mentioned above, Addario-Berry and Maillard~\cite{AB&M20} considered this optimization problem for the continuous random energy model (CREM), which is a binary time-inhomogeneous branching random walk with Gaussian displacements, proving the existence of an threshold $x_*$ such that the following holds: a) for every $x<x_*$, there exists a polynomial-time algorithm that finds a vertex with $X_v \ge xN$ with high probability, b) for every $x>x_*$, every such algorithm has a running time which is at least exponential in $N$ with high probability. 

The CREM, introduced by Bovier and Kurkova~\cite{CREM04} based on previous work by Derrida and Spohn~\cite{DerridaSpohn}, is a toy model of a disordered system in statistical physics, i.e.~a model where the Hamiltonian -- the function that assigns energies to the states of the system -- is itself random. These systems have recently seen a lot of interest in the mathematical literature with regards to efficient algorithms for finding low-energy states. A key quantity of importance in these models is the so-called \emph{overlap} between two states, a measure of their correlation. In the case of the CREM, it is equal to the depth of the most recent common ancestor of two vertices, divided by $N$. Then, for a given $\beta>0$, the \emph{overlap distribution} is the limiting law (as $N\to\infty$) of the overlap of two vertices sampled independently according to the Gibbs measure with inverse temperature $\beta$. A picture that has emerged is that the existence of a gap in the support of the overlap distribution, the so-called ``overlap gap property'', is an obstruction to the existence of efficient algorithms finding approximate minimizers of the Hamiltonian. This has been rigorously proven for a certain class of algorithms by Gamarnik and Jagannath \cite{OGPAMP21} in the case of the Ising $p$-spin model, with $p\geq 4$. On the other hand, Montanari \cite{Montanari19} showed that for $p=2$, the Sherrington--Kirkpatrick model, there exists a quadratic time algorithm that can find near optimal state with high probability, assuming a widely believed conjecture that this model does not exhibit an overlap gap. A similar result for spherical spin glass models (for which the overlap distribution is explicitly known) has been obtained by Subag \cite{Subag21}.

The question of efficient \emph{sampling} of the Gibbs measure of a disordered system seems to have been considered mostly under the angle of Glauber or Langevin dynamics. See e.g. \cite{SpecGap18,GJ19} for the spherical spin glass model. We restrict ourselves here to the case of the Sherrington--Kirkpatrick model. For this model, it has been recently obtained that fast mixing occurs for $\beta < 1/4$ \cite{BB19,EKZ21}. However, very recently, El Alaoui, Montanari and Sellke \cite{EAMS22} have provided another algorithm which yields fast mixing for $\beta<1/2$, and they conjecture that this in fact holds for all $\beta<1$. They also provide a hardness result for $\beta >1$ for a certain class of algorithms. Their algorithm for the $\beta < 1/2$ phase belongs to the class of \emph{approximate message passing} algorithms, which is also the case for Montanari's algorithm for the optimization problem~\cite{Montanari19}. This illustrates the fact that Glauber or Langevin dynamics may in general not be optimal sampling algorithms, and that algorithms which exploit the underlying tree structure of the model may be efficient in a wider range of the parameters. For a discussion of this question in the context of statistical inference problems, see e.g.~\cite{GlassHardInfer19}. Altogether, this motivates the study of tree-based models as a toy problem, such as the one from the present article.

\paragraph{Outline.} The paper is organized as follows. In Section~\ref{sec:decomp}, we prove that the Kullback--Leibler divergence can be decomposed into a weighted sum of the Kullback--Leibler divergences on subtrees. In Section~\ref{sec:Lp}, we give $L^p$ bounds of the logarithm of Biggins' martingales and $L^p$ bounds of the Kullback--Leibler divergences between two Gibbs measures. Theorem~\ref{thm:main} is proven in Section~\ref{sec:proof main thm} and Theorem~\ref{th:hardness} in Section~\ref{sec:proof hardness}. Section~\ref{sec:proof_lemma} provides the proof of Lemma~\ref{lem:entropy}. Finally, we state in Section~\ref{sec:open questions} some open questions that might interest the readers.

\section{Decomposition of the Kullback--Leibler divergence} \label{sec:decomp}

The main goal of this section is to prove Theorem \ref{thm:KL decomp}. Before proving the theorem, we need the following lemma, which states that the weight of $u_1u_2$ with respect to the Gibbs measure $\mu_{\beta,m}$ can be decomposed into the product of the weights of $u_1$ and $u_2$ with respect to another two Gibbs measures.

\begin{lemma} \label{lem:gibbsprod}
  For any $\abs{u_1}=m_1$ and $\abs{u_2}=m_2$, we have the decomposition
  \begin{align*}
    \mu_{\beta,m}(u_1u_2) = \mu_{\beta,m}(u_1)\cdot \mu_{\beta,m-m_1}^{u_1}(u_2)
  \end{align*}
  for all $m\geq m_1+m_2$.
\end{lemma}

\begin{proof}
  By computation, we have
  \begin{align*}
    \mu_{\beta,m}(u_1u_2)
    & = \frac{e^{\beta X_{u_1u_2}-\varphi(\beta)m}\cdot W_{\beta,m-m_1-m_2}^{u_1u_2}}{W_{\beta,m}} \\
    & = \frac{e^{\beta X_{u_1}-\varphi(\beta)m_1}\cdot W_{\beta,m-m_1}^{u_1}}{W_{\beta,m}}
    \cdot \frac{e^{\beta X_{u_2}^{u_1}-\varphi(\beta)(m-m_1)}\cdot W_{\beta,m-m_1-m_2}^{u_1u_2}}{W_{\beta,m-m_1}^{u_1}} \\
    & = \mu_{\beta,m}(u_1)\cdot \mu_{\beta,m-m_1}^{u_1}(u_2),
  \end{align*}
  where the last equality is by \eqref{def:Gibbs0} and \eqref{def:Gibbs1}.
\end{proof}

Now we can decompose the Kullback--Leibler divergence as follows.

\begin{theorem} \label{thm:KL decomp}
  For any two $M$ and $N$ integers such that $M\leq N$, we have
  \begin{align}
    & \kld{\mu_{\beta,M,N}}{\mu_{\beta,N}}
    = \sum_{K=0}^{\floor{\frac{N}{M}}-1} \sum_{\abs{u}=KM}\mu_{\beta,M,KM}(u)\cdot \kld{\mu_{\beta,M}^u}{\mu_{\beta,N-KM}^u}. \label{eq:KL decomp.2}
  \end{align}
\end{theorem}

\begin{proof}
  Denote $N'=\floor{\frac{N}{M}}\cdot M$ for simplicity. 
  By \eqref{eq:mu_decomposition} and Lemma \ref{lem:gibbsprod}, for all $\abs{u_1}=N'$ and $\abs{u_2}=N-N'$, we have
  \begin{align}
    \log\left(\frac{\mu_{\beta,M,N}(u_1u_2)}{\mu_{\beta,N}(u_1u_2)}\right)
    = \log\left(\frac{\mu_{\beta,M,N'}(u_1)\cdot \mu_{\beta,N-N'}^{u_1}(u_2)}{\mu_{\beta,N}(u_1)\cdot \mu_{\beta,N-N'}^{u_1}(u_2)}\right) 
    =\log\left(\frac{\mu_{\beta,M,N'}(u_1)}{\mu_{\beta,M,N}(u_1)}\right). \label{eq:N->N'}
  \end{align}
  Thus, the Kullback--Leibler divergence can be rewritten as
  \begin{align}
    &\kld{\mu_{\beta,M,N}}{\mu_{\beta,N}} \nonumber \\
    &= \sum_{\abs{u}=N} \mu_{\beta,M,N}(u)\cdot \log\left(\frac{\mu_{\beta,M,N}(u)}{\mu_{\beta,N}(u)}\right) \nonumber \\
    &= \sum_{\abs{u_1}=N'}\sum_{\abs{u_2}=N-N'} \mu_{\beta,M,N}(u_1u_2) \cdot \log\left(\frac{\mu_{\beta,M,N}(u_1u_2)}{\mu_{\beta,N}(u_1u_2)}\right) \nonumber \\
    &= \sum_{\abs{u_1}=N'}\sum_{\abs{u_2}=N-N'} \mu_{\beta,M,N'}(u_1)\cdot \mu_{\beta,N-N'}^{u_1}(u_2) \cdot \log\left(\frac{\mu_{\beta,M,N'}(u_1)}{\mu_{\beta,N}(u_1)}\right) && \text{(by \eqref{eq:mu_decomposition} and \eqref{eq:N->N'})} \nonumber \\
    & = \sum_{\abs{u_1}=N'} \mu_{\beta,M,N'}(u_1)\cdot \log\left(\frac{\mu_{\beta,M,N'}(u_1)}{\mu_{\beta,N}(u_1)}\right) \nonumber \\
    &= \kld{\mu_{\beta,M,N'}}{\mu_{\beta,N}}. \label{eq:KL N' N}
  \end{align}

  Next, by \eqref{eq:mu_decomposition} and Lemma~\ref{lem:gibbsprod}, for all $\abs{u_1}=KM$, $\abs{u_2}=M$ and $0\leq K\leq \floor{\frac{N}{M}}-1$, we have
  \begin{align}
    \log\left(\frac{\mu_{\beta,M,(K+1)M}(u_1u_2)}{\mu_{\beta,N}(u_1u_2)}\right)
    &= \log\left(\frac{\mu_{\beta,M,KM}(u_1)\cdot \mu_{\beta,M}^{u_1}(u_2)}{\mu_{\beta,N}(u_1)\cdot \mu_{\beta,N-KM}^{u_1}(u_2)}\right)  \nonumber \\
    &=\log\left(\frac{\mu_{\beta,M,KM}(u_1)}{\mu_{\beta,N}(u_1)}\right)+\log\left(\frac{\mu_{\beta,M}^{u_1}(u_2)}{\mu_{\beta,N-KM}^{u_1}(u_2)}\right). \label{eq:N->N'.2}
  \end{align}
  Thus,
  \begin{align}
    &\kld{\mu_{\beta,M,(K+1)M}}{\mu_{\beta,N}} \nonumber \\
    &=\sum_{\abs{u}=(K+1)M}\mu_{\beta,M,(K+1)M}(u)\cdot \log\left(\frac{\mu_{\beta,M,(K+1)M}(u)}{\mu_{\beta,N}(u)}\right) \nonumber \\
    &=\sum_{\abs{u_1}=KM}\sum_{\abs{u_2}=M}\mu_{\beta,M,KM}(u_1) \nonumber \\
    &\qquad\cdot \mu_{\beta,M}^{u_1}(u_2)\cdot \left[\log\left(\frac{\mu_{\beta,M,KM}(u_1)}{\mu_{\beta,N}(u_1)}\right)+\log\left(\frac{\mu_{\beta,M}^{u_1}(u_2)}{\mu_{\beta,N-KM}^{u_1}(u_2)}\right)\right] && \text{(by \eqref{eq:mu_decomposition} and \eqref{eq:N->N'.2})} \nonumber \\
    &= \kld{\mu_{\beta,M,KM}}{\mu_{\beta,N}} 
    + \sum_{\abs{u_1}=KM}\mu_{\beta,M,KM}(u_1)\cdot \kld{\mu_{\beta,M}^{u_1}}{\mu_{\beta,N-KM}^{u_1}}.\label{eq:backward induct}
  \end{align}
  
  Finally, by \eqref{eq:KL N' N}, \eqref{eq:backward induct} and the fact that $\kld{\mu_{\beta,M,0}}{\mu_{\beta,N}}=0$, we derive \eqref{eq:KL decomp.2}. 
\end{proof}

Next, we show that the Kullback--Leibler
divergence between two Gibbs measures can be written in terms of the logarithms of the partition functions.

\begin{proposition} \label{prop:KL decomp}
  For any two $M$ and $N$ integers such that $M\leq N$, we have
  \begin{align}
    & \kld{\mu_{\beta,M}}{\mu_{\beta,N}} = \log W_{\beta,N} - \log W_{\beta,M} - \sum_{\abs{u}=M}\mu_{\beta,M}(u)\cdot \log W^u_{\beta,N-M}.
  \end{align}
\end{proposition}
\begin{proof}
  For any $\abs{u}=M$, we have
  \begin{align*}
    \frac{\mu_{\beta,M}(u)}{\mu_{\beta,N}(u)}
    = \frac{e^{\beta X_u-\varphi(\beta)M}\cdot \frac{1}{W_{\beta,M}}}{e^{\beta X_u-\varphi(\beta)M}\cdot \frac{W_{\beta,N-M}^u}{W_{\beta,N}}} 
    = \frac{W_{\beta,N}}{W_{\beta,M}\cdot W^u_{\beta,N-M}}.
  \end{align*}
  Thus,
  \begin{align*}
    \kld{\mu_{\beta,M}}{\mu_{\beta,N}}
    & = \sum_{\abs{u}=M} \mu_{\beta,M}(u)\cdot \log\left(\frac{\mu_{\beta,M}(u)}{\mu_{\beta,N}(u)}\right) \\
    & = \sum_{\abs{u}=M} \mu_{\beta,M}(u) \cdot
    \left(
    \log W_{\beta,N} - \log W_{\beta,M} - \log W_{\beta,N-M}^u
    \right) \\
    & = \log W_{\beta,N} - \log W_{\beta,M} - \sum_{\abs{u}=M}\mu_{\beta,M}(u)\cdot \log W^u_{\beta,N-M}.
  \end{align*}
  This completes the proof.
\end{proof}

\section{Some \texorpdfstring{$L^p$}{Lp} bounds} \label{sec:Lp}

We first show that whenever $\beta\in [0,\beta_c)$, $(\log W_{\beta,n})_{n\geq 0}$ is bounded in $L^p$ for all $p>1$. 

\begin{lemma} \label{lem:log W}
  Let $\beta\in [0,\beta_c)$. Then $(\log W_{\beta,n})_{n\geq 0}$ is a supermartingale such that \[\sup_{n\geq 0}\norm{\log W_{\beta,n}}_p < \infty, \quad p\geq1.\]
\end{lemma}
\begin{proof}
  The supermartingale property of $(\log W_{\beta,n})_{n\geq0}$ follows from the fact that $(W_{\beta,n})_{n\geq 0}$ is a martingale and $x\mapsto\log x$ is a concave function.

  Now let $p \geq 1$. Since $\beta<\beta_c$, by Fact~\ref{fact:W}, we derive that
  \[\sup_{n\geq 0}\norm{W_{\beta,n}}_1 < \infty\]
  Furthermore, using Assumption~\ref{assump:phi} and Liu \cite[Theorem 2.4]{Liu01}, we have for some $s>0$, 
  \[\sup_{n\geq 0}\|W_{\beta,n}^{-s}\|_1 < \infty.\]  
 Then by the fact that
  \begin{align*}
  \abs{\log x}^p\leq C\left(\abs{x}+\abs{x}^{-s}\right)
  \end{align*}
  for some constant $C>0$, we have
  \begin{align*}
    \sup_{n\geq 0}\norm{\log W_{\beta,n}}_p^p\leq  C\cdot\sup_{n\geq 0}\left(\norm{W_{\beta,n}}_1+\|W_{\beta,n}^{-s}\|_1\right) < \infty.
  \end{align*}
  This proves the lemma.
\end{proof}

Lemma \ref{lem:log W} implies the following proposition about the boundedness of the Kullback--Leibler divergence between two Gibbs measures.

\begin{proposition}
\label{prop:3.3}
For any $p>1$, for any two integers $M$ and $N$ such that $M\leq N$, there exists a constant $C(p)>0$ such that
\[\norm{\kld{\mu_{\beta,M}}{\mu_{\beta,N}}}_p\leq C(p).\]
\end{proposition}
\begin{proof}
By Minkowski's inequality and Proposition~\ref{prop:KL decomp}, we have
\begin{align*}
\norm{\kld{\mu_{\beta,M}}{\mu_{\beta,N}}}_p\leq  \norm{\log W_{\beta,N}}_p + \norm{\log W_{\beta,M}}_p + \norm{\sum_{\abs{u}=M}\mu_{\beta,M}(u)\cdot \log W^u_{\beta,N-M}}_p.
\end{align*}
Since the first and the second term above are bounded by Lemma~\ref{lem:log W}, it suffices to prove that the third term above is bounded. By the branching property and Jensen's inequality, we have 
\begin{align}
\cEX{\abs{\sum_{\abs{u}=M}\mu_{\beta,M}(u)\cdot \log W^u_{\beta,M}}^p}{\F_M}
&\leq \EX{\abs{\log W_{\beta,N-M}}^p}. \label{eq:lem 3.3.1} 
\end{align}
Therefore, by \eqref{eq:lem 3.3.1},
\begin{align}
\norm{\sum_{\abs{u}=M}\mu_{\beta,M}(u)\cdot \log W^u_{\beta,N-M}}_p
\leq \norm{\log W_{\beta,N-M}}_p. \label{eq:lem 3.3.2}
\end{align}
Finally, we conclude by \eqref{eq:lem 3.3.2} and Lemma~\ref{lem:log W} that
\begin{align*}
\norm{\kld{\mu_{\beta,M}}{\mu_{\beta,N}}}_p \leq 3\cdot \sup_{n\geq 0}\norm{\log W_{\beta,n}}_p < \infty,
\end{align*}
and the proof is completed.
\end{proof}

\section{Proof of Theorem \ref{thm:main}} \label{sec:proof main thm}

In this section, we prove \eqref{eq:main.1} and \eqref{eq:main.2} of Theorem~\ref{thm:main}. The proof of \eqref{eq:main.1} relies essentially on the decomposition theorem of the Kullback--Leibler divergence (Theorem~\ref{thm:KL decomp}) and Proposition~\ref{prop:3.3}. The proof of \eqref{eq:main.2} needs more precise moment estimates.

\subsection{Proof of (\ref{eq:main.1})} \label{sec:proof thm main.1}

Let $p\ge1$. By Theorem~\ref{thm:KL decomp} and Minkowski's inequality, we have
\begin{align}
\norm{\kld{\mu_{\beta,M,N}}{\mu_{\beta,N}}}_p
&= \norm{\sum_{K=0}^{\floor{\frac{N}{M}}-1} \sum_{\abs{u}=KM}\mu_{\beta,M,KM}(u)\cdot \kld{\mu_{\beta,M}^u}{\mu_{\beta,N-KM}^u}}_p \nonumber \\
&\leq \sum_{K=0}^{\floor{\frac{N}{M}}-1}\norm{ \sum_{\abs{u}=KM}\mu_{\beta,M,KM}(u)\cdot \kld{\mu_{\beta,M}^u}{\mu_{\beta,N-KM}^u}}_p. \label{eq:main pf.1}
\end{align}

Let $K\le \lfloor N/M\rfloor-1$. Applying Jensen's inequality to $\mu_{\beta,M,KM}$, we have
\begin{align}
&\EX{\abs{\sum_{\abs{u}=KM}\mu_{\beta,M,KM}(u)\cdot \kld{\mu_{\beta,M}^u}{\mu_{\beta,N-KM}^u}}^p} \nonumber \\
&\leq \EX{\sum_{\abs{u}=KM}\mu_{\beta,M,KM}(u)\cdot \abs{\kld{\mu_{\beta,M}^u}{\mu_{\beta,N-KM}^u}}^p}. \label{eq:main pf.2}
\end{align}
Then by the law of iterated expectation and the branching property, \eqref{eq:main pf.2} is equal to
\begin{align}
&\EX{\sum_{\abs{u}=KM}\mu_{\beta,M,KM}(u)\cdot \cEX{\abs{\kld{\mu_{\beta,M}^u}{\mu_{\beta,N-KM}^u}}^p}{\F_{KM}}} \nonumber \\
&=\EX{\sum_{\abs{u}=KM}\mu_{\beta,M,KM}(u)}\cdot \EX{\abs{\kld{\mu_{\beta,M}}{\mu_{\beta,N-KM}}}^p} = \EX{\abs{\kld{\mu_{\beta,M}}{\mu_{\beta,N-KM}}}^p}. \label{eq:main pf.3}
\end{align}
Combining \eqref{eq:main pf.1}, \eqref{eq:main pf.3} and Proposition~\ref{prop:3.3}, we conclude that 
\begin{align*}
\norm{\kld{\mu_{\beta,M,N}}{\mu_{\beta,N}}}_p
\leq \floor*{\frac{N}{M}}\cdot C(p).
\end{align*}

\subsection{Proof of (\ref{eq:main.2})} \label{sec:proof 1.5}

In this section, we prove \eqref{eq:main.2} which gives a tighter control on the Kullback--Leibler divergence between $\mu_{\beta,M,N}$ and the Gibbs measure $\mu_{\beta,N}$. We start with the following simple lemma.

\begin{lemma} \label{lem:sum pth power}
Let $\beta\in [0,\beta_c)$. For all $M\in\N$, there exists $r\in (0,1)$ independent of $M$ such that,
  \begin{align}
    &\EX{\sum_{\abs{u}=M}\mu_{\beta,M}(u)^2}  \le r \label{eq:sum pth power.1}
  \end{align}
Moreover, for all $K\in\N$,
  \begin{align}
    &\EX{\sum_{\abs{u}=KM}\mu_{\beta,M,KM}(u)^2} \le r^K. \label{eq:sum pth power.2}
  \end{align}
\end{lemma}
\begin{proof}[Proof of Lemma \ref{lem:sum pth power}]
  For $x\in (0,1)$, $x^2<x$. Thus, by the fact that $\mu_{\beta,M}(u)\in (0,1)$ for all $\abs{u}=M$, we derive that
  \begin{align*}
    \sum_{\abs{u}=M} \mu_{\beta,M}(u)^2 < \sum_{\abs{u}=M} \mu_{\beta,M}(u) = 1,
  \end{align*}
  for every $M\in\N$. This shows that \eqref{eq:sum pth power.1} holds for every fixed $M\in\N$ and with $r<1$ possibly depending on $M$. Uniformity in $M$ follows as soon as we  show that \[\limsup_{M\to\infty} \EX{\sum_{\abs{u}=M}\mu_{\beta,M}(u)^2} < 1.\]
  
  To this end, recall that $W_{\beta,M} \to W_{\beta,\infty}$ almost surely as $M\to\infty$ and that $\E[W_{\beta,\infty}] = 1$. Hence, there exist $a<1$ and $M_0\in\N$ such that 
  \begin{equation}
  \label{eq:WbetaMr}
  \forall M\ge M_0: \PP(W_{\beta,M} < 1/2) \le a.
  \end{equation}
  Now fix $p\in(1,2]$ such that $\varphi(p\beta) < p\varphi(\beta)$, which exists because $\beta < \beta_c$. Then decompose:
  \begin{align*}
  \EX{\sum_{\abs{u}=M}\mu_{\beta,M}(u)^2} 
  &\le \EX{\sum_{\abs{u}=M}\mu_{\beta,M}(u)^p}\\
  &\le \EX{\left(\sum_{\abs{u}=M}\mu_{\beta,M}(u)^p\right)\Ind_{W_{\beta,M} < 1/2}}+ \EX{\left(\sum_{\abs{u}=M}\mu_{\beta,M}(u)^p\right)\Ind_{W_{\beta,M} \ge 1/2}}\\
  &\le \PP(W_{\beta,M} < 1/2) + 2^p e^{(\varphi(p\beta) - p\varphi(\beta))M},
  \end{align*}
  using the definition of $\mu_{\beta,M}(u)$ for the last inequality. This shows that \[\limsup_{M\to\infty} \EX{\sum_{\abs{u}=M}\mu_{\beta,M}(u)^2} \le a,\] 
  which concludes the proof of \eqref{eq:sum pth power.1}.

  We now show \eqref{eq:sum pth power.2} by induction. The case $K=1$ follows directly from \eqref{eq:mu_decomposition} and \eqref{eq:sum pth power.1}.
  For $K>1$, by \eqref{eq:mu_decomposition} and branching property, we obtain that
  \begin{align}
    &\cEX{\sum_{\abs{u}=KM}\mu_{\beta,M,KM}(u)^2}{\F_{(K-1)M}}\\
    &= \cEX{\sum_{\abs{u_1}=(K-1)M}\sum_{\abs{u_2}=M} \mu_{\beta,M,(K-1)M}(u_1)^2\cdot \mu_{\beta,M}^{u_1}(u_2)^2}{\F_{(K-1)M}} \nonumber \\
    &= \sum_{\abs{u_1}=(K-1)M} \mu_{\beta,M,(K-1)M}(u_1)^2\cdot \EX{\sum_{\abs{u_2}=M}\mu_{\beta,M}(u_2)^2}. \label{eq:mainlem2}
  \end{align}
Taking expectations, we derive that
  \begin{align}
    \EX{\sum_{\abs{u}=KM}\mu_{\beta,M,KM}(u)^2} = \EX{\sum_{\abs{u}=(K-1)M}\mu_{\beta,M,(K-1)M}(u)^2}\cdot \EX{\sum_{\abs{u}=M}\mu_{\beta,M}(u)^2}.
  \end{align}
  The equation \eqref{eq:sum pth power.2} then follows from the induction hypothesis.
\end{proof}

We now proceed with the proof of \eqref{eq:main.2}. Without loss of generality, using the fact that $\|\cdot\|_p \le \|\cdot\|_2$ for every $p\in[1,2]$, we assume that $p\ge2$. By Theorem~\ref{thm:KL decomp} and Minkowski's inequality, we have 
\begin{align}
&\norm{\kld{\mu_{\beta,M,N}}{\mu_{\beta,N}}-\EX{\kld{\mu_{\beta,M,N}}{\mu_{\beta,N}}}}_p \nonumber \\
&\leq 
\sum_{K=0}^{\floor{\frac{N}{M}}-1} \norm{\sum_{\abs{u}=KM}\mu_{\beta,M,KM}(u)\cdot \left(\kld{\mu_{\beta,M}^u}{\mu_{\beta,N-KM}^u}-\EX{\kld{\mu_{\beta,M}^u}{\mu_{\beta,N-KM}^u}}\right)}_p. \label{eq:uppbnd}
\end{align}

We now introduce some notation. For all $0\leq K\leq \floor{\frac{N}{M}}-1$, denote
\begin{align*}
d_K^u = \kld{\mu_{\beta,M}^u}{\mu_{\beta,N-KM}^u} \quad\text{and}\quad
d_K = \kld{\mu_{\beta,M}}{\mu_{\beta,N-KM}}
\end{align*}
and 
\begin{align*}
Z_u = \mu_{\beta,M,KM}(u)\cdot \left(d_K^u-\EX{d_K^u}\right).
\end{align*}

We claim that for all $p\geq 1$, the sequence
\begin{align}
\label{eq:claim}
a_K = \norm{\sum_{\abs{u}=KM}Z_u}_p
\end{align}
is summable, with a bound independent of $M$. This will imply that the right-hand side of \eqref{eq:uppbnd} is bounded by the same quantity, which completes proof. To prove this, first observe that $(Z_u)_{\abs{u}=KM}$ is a sequence of iid random variables having zero mean and finite $p$-th moments, for any $p\geq 1$, with respect to $\cEX{\,\cdot}{\F_{KM}}$, by Proposition~\ref{prop:3.3} and the branching property. Denote by $C_1,C_2,\ldots$ some constants possibly depending on $p$ (and $\beta$ and the law of $\Y$). By Rosenthal's inequality \cite[Theorem 3]{Rosenthal70}, we have
\begin{align}
& \cEX{\abs{\sum_{\abs{u}=KM} Z_u}^p}{\F_{KM}} 
\leq C_1\cdot \left\{
\left(\sum_{\abs{u}=KM}\cEX{\abs{Z_u}^p}{\F_{KM}}\right)+
\left(\sum_{\abs{u}=KM}\cEX{\abs{Z_u}^2}{\F_{KM}}\right)^{p/2}\right\} \label{eq:main pf.6}
\end{align}

By the branching property and Proposition~\ref{prop:3.3},
the first term of \eqref{eq:main pf.6} can be bounded by
\begin{align}
\sum_{\abs{u}=KM}\cEX{\abs{Z_u}^p}{\F_{KM}}
&= \sum_{\abs{u}=KM}\mu_{\beta,M,KM}(u)^p\cdot \EX{\abs{d_K-\EX{d_K}}^p} \nonumber \\
&\leq \sum_{\abs{u}=KM}\mu_{\beta,M,KM}(u)^p\cdot C_2 \nonumber \\
&\leq \sum_{\abs{u}=KM}\mu_{\beta,M,KM}(u)^2\cdot C_2, \label{eq:main pf.6.0}
\end{align}
using that $p\ge2$ in the last line.
Taking expectations and applying Lemma~\ref{lem:sum pth power}, we get 
\begin{align}
\EX{\sum_{\abs{u}=KM}\cEX{\abs{Z_u}^p}{\F_{KM}}}
\leq \EX{\sum_{\abs{u}=KM}\mu_{\beta,M,KM}(u)^2}\cdot C_2
\leq r^K\cdot C_2, \label{eq:main pf.6.0.1}
\end{align}
with $r<1$ as in Lemma~\ref{lem:sum pth power}.

We now estimate the second term of \eqref{eq:main pf.6}. By the branching property and Lemma~\ref{lem:log W}, 
\begin{align}
\left(\sum_{\abs{u}=KM}\cEX{\abs{Z_u}^2}{\F_{KM}}\right)^{p/2} 
&= \left(\sum_{\abs{u}=KM}\mu_{\beta,M,KM}(u)^2\cdot \EX{\abs{d_K-\EX{d_K}}^2}\right)^{p/2} \nonumber \\
&\leq \left(\sum_{\abs{u}=KM}\mu_{\beta,M,KM}(u)^2\right)^{p/2}\cdot C_3 \nonumber \\
&\leq \left(\sum_{\abs{u}=KM}\mu_{\beta,M,KM}(u)^2\right)\cdot C_3. \label{eq:main pf.6.1}
\end{align}
The inequality \eqref{eq:main pf.6.1} is because 
\[\sum_{\abs{u}=KM}\mu_{\beta,M,KM}(u)^2\in [0,1]\] 
and $x^{p/2} \leq x$, for all $x\in [0,1]$ and $p \ge 2$. Taking expectations and applying Lemma~\ref{lem:sum pth power}, we get 
\begin{align}
\EX{\left(\sum_{\abs{u}=KM}\cEX{\abs{Z_u}^2}{\F_{KM}}\right)^{p/2}}
\leq \EX{\sum_{\abs{u}=KM}\mu_{\beta,M,KM}(u)^2}\cdot C_3
\leq r^K\cdot C_3. \label{eq:main pf.6.2}
\end{align}

Combining \eqref{eq:main pf.6}, \eqref{eq:main pf.6.0.1} and \eqref{eq:main pf.6.2}, we conclude that
\begin{align*}
\norm{\sum_{\abs{u}=KM} Z_u}_p \leq C_1(C_2+C_3)\cdot r^K.
\end{align*}
This implies that \eqref{eq:claim} is summable and finishes the proof.

\section{Proof of hardness result (Theorem~\ref{th:hardness})}
\label{sec:proof hardness}

In this section, we prove Theorem~\ref{th:hardness}. To do this, we recall two results of asymptotic behaviors of the maximal particle of a branching random walk. 

Under the assumption of the theorem, it is known in Corollary~(3.4) of \cite{BigginsChernoff} that
\begin{equation}
\label{eq:max}
\frac{\max_{|u|=N} X_u}{N} \to m \coloneqq \varphi'(\beta_c),\quad \text{a.s. as $N\to\infty$}.
\end{equation}
Furthermore, we have the following tail estimate, which easily follows from a union bound together with Chernoff's bound (see e.g.~the proof of Theorem~2 in \cite{ZeitouniLNBRW}): there exists a constant $c>0$ such that 
\begin{equation}
\label{eq:max_tail}
\forall x\ge 0: \mathbb P\left(\max_{|u|=N} X_u \ge mN + x\right) \le e^{-c x}.
\end{equation}

The key to Theorem~\ref{th:hardness} is the following observation:

\begin{lemma}
\label{lem:sqrt}
Let $\beta > \beta_c$ and assume $\beta_c\in \mathcal D(\varphi)^\circ$. Let $u$ be a particle sampled according to the Gibbs measure $\mu_{\beta,N}$ and let $w$ be its ancestor at generation $\floor{N/2}$. Then there exists a positive random variable $Z$ with continuous distribution function such that
\[
\frac{X_u - X_{w} - mN/2}{\sqrt{N}} \to Z,\quad\text{in law as $N\to\infty$}.
\]
\end{lemma}

\begin{proof}
This is a consequence of a result by Chen, Madaule and Mallein \cite{ChenMadauleMallein}. These authors show the following fact: if $u$ is sampled according to the Gibbs measure $\mu_{\beta,N}$, and $X_u(t)$ denotes the position of its ancestor at generation $\lfloor tN\rfloor$, and if we define
\[
Z^N_t \coloneqq \frac{mtN - X_u(t)}{\sqrt N},\quad t\in[0,1],
\]
then $(Z^N_t)_{t\in[0,1]}$ converges in law (w.r.t. Skorokhod's topology) to a multiple of a Brownian excursion as $N\to\infty$. Note that the assumptions in their article are implied by our hypothesis that $\beta_c\in \mathcal D(\varphi)^\circ$ and the fact that in our branching random walk, the number of offspring of a particle is deterministic. Now, we also have that $X_u - mN = O(\log N)$ in probability (see e.g.~\cite{Aidekon13}), so that
\[
\frac{X_u - mN}{\sqrt N} \xrightarrow{\text{law}} 0,\quad \text{as $N\to\infty$}.
\]
Together, both results imply the lemma.
\end{proof}

\begin{proof}[Proof of Theorem~\ref{th:hardness}]
Assume that we are given an algorithm $(v(k))_{k\ge1}$ that samples a vertex according to a random probability measure $\tilde \mu_{\beta,N}$ approximating the Gibbs measure $\mu_{\beta,N}$. By Lemma~\ref{lem:entropy}, it follows that $\kld{\tilde \mu_{\beta,N}}{\mu_{\beta,N}} \to 0$ in probability as $N\to\infty$. Hence, by Pinsker's inequality (see e.g. Theorem~4.19~in~\cite{ConIneq13}), the total variation distance between $\tilde \mu_{\beta,N}$ and $\mu_{\beta,N}$ goes to 0 as well in probability, as $N\to\infty$. It follows that Lemma~\ref{lem:sqrt} holds as well for $u$ sampled according to $\tilde \mu_{\beta,N}$.

Let $\delta > 0$. For $z>0$, call a vertex $w \in \partial \T_{\lfloor N/2\rfloor}$ \emph{$z$-exceptional} if it has a descendant $u\in\partial \T_N$ such that $X_u - X_{w} - mN/2 > z\sqrt N$. By the preceding paragraph, there exists $z > 0$ such that for large enough $N$, the algorithm finds a vertex $u$ whose ancestor $w$ at generation $\lfloor N/2\rfloor$ is $z$-exceptional with probability at least $1-\delta$. Hence, it is enough to show that any algorithm which solves the simpler problem of finding a $z$-exceptional vertex at generation $\lfloor N/2\rfloor$ has a running time at least $e^{z'\sqrt{N}}$ with probability $1-\delta$, for some $z'>0$. We will now show that the statement of the theorem holds even for this simpler problem. For this, we use an argument similar to the one in Section~3 of \cite{AB&M20}. We first present the argument in an informal way.

Denote by $E_w$ the event that a given vertex $w\in \partial \T_{\lfloor N/2\rfloor}$ is $z$-exceptional. Note that this event only depends on the displacements of the descendants of $w$. Hence, the events $(E_w)_{w\in \partial \T_{\lfloor N/2\rfloor}}$ are independent by the branching property. Furthermore, by \eqref{eq:max_tail}, for each $w\in \partial \T_{\lfloor N/2\rfloor}$, we have $\PP(E_w) \le e^{-cz\sqrt{N/2}}$ for some $c>0$. Finally, in order to determine whether a vertex $w$ is $z$-exceptional, the algorithm has to explore at least one vertex in the subtree of the vertex $w$. Hence, the running time of the algorithm is bounded from below by the number of vertices $w$ that have to be probed in order to find a $z$-exceptional vertex. But this quantity follows the geometric distribution with success probability $\PP(E_w)\le e^{cz\sqrt{N/2}}$. Altogether, for any $z'<cz/\sqrt 2$ and for $N$ sufficiently large, this shows that the running time $\tau$ of any algorithm solving the simpler problem is at least $e^{z'\sqrt{N}}$ with probability $1-\delta$. The statement readily follows.

We now make this argument formal. Recall that, by definition, an algorithm is a stochastic process $(v(n))_{n\ge0}$ previsible with respect to the filtration $\tilde{\F}$, defined by
\[
\tilde{\F}_k = \sigma\left(v(1),\ldots,v(k);\,X_{v(1)},\ldots,X_{v(k)};\,U_1,\ldots,U_{k+1}\right)
\]
where $(U_k)_{k\geq 1}$ is a sequence of iid uniform random variables on $[0,1]$, independent of the branching random walk $\mathbf{X}$. We now define a larger filtration $\G$. For this, define for any $v\in\T_N$ the following set of vertices:
\[
\mathcal V_v = \begin{cases}
w\in \T_N: |v\wedge w| \ge \lfloor N/2\rfloor,& \text{if $|v|\ge \lfloor N/2\rfloor$}\\
v, & \text{otherwise}.
\end{cases}
\]
Note that $v\in \mathcal V_v$ for every $v\in \T_N$.
We then set
\[
\G_k = \sigma\left(v(1),\ldots,v(k);\,(X_{w})_{w\in\mathcal V_{v(1)}},\ldots,(X_{w})_{w\in\mathcal V_{v(k)}};\,U_1,\ldots,U_{k+1}\right).
\]
Note that $\tilde{\F}_k \subset \G_k$ for all $k\ge0$ --- heuristically, $\G_k$ adds to $\tilde{\F}_k$ the information about the values in the branching random walk of all vertices contained in $\mathcal V_{v(i)}$, $i=1,\ldots,k$. Note that trivially, the stochastic process $v(n)_{n\ge0}$ is still previsible with respect to this larger filtration $\G$.

Now say that $\mathcal V_v$ is $z$-exceptional if $|v|\ge \lfloor N/2\rfloor$ and the ancestor of $v$ at generation $\lfloor N/2\rfloor$ is $z$-exceptional in the sense defined above --- note that this definition does not depend on the choice of $v$. Define
\[
\tau' = \inf\{k\ge 0: \mathcal V_{v(k)}\text{ is $z$-exceptional}\},
\]
and note that $\tau'$ is a stopping time with respect to the filtration $\G$.
Now, by the equality of  events
\[
\{\tau' = k\} = \{\text{$\mathcal V_{v(k)}$ is $z$-exceptional}\}\cap \{\tau'>k-1\},
\]
and since $\{\tau'>k-1\}\in\G_{k-1}$ and $v(k)$ is $\G_{k-1}$-measurable, we have
\[
\PP(\tau'=k\,|\,\G_{k-1}) = \sum_{v\in\T_N} \PP(\text{$\mathcal V_v$ is $z$-exceptional}\,|\,\G_{k-1})\Ind_{(v(k) = v,\ \tau' > k-1)}.
\]
Now, for any $v\in\T_N$, if $|v| < \lfloor N/2\rfloor$, or if $v \in \bigcup_{i=0}^{k-1} \mathcal V_{v(i)}$, the above probability is zero, because none of $\mathcal V_{v(i)}$, $i=0,\ldots,k-1$ are $z$-exceptional on the event $\{\tau' > k-1\}$. On the other hand, if $v\not\in \bigcup_{i=0}^{k-1} \mathcal V_{v(i)}$, then, by the branching property, the above conditional probability is equal to the unconditioned probability that $\mathcal V_v$ is $z$-exceptional, which is bounded by $e^{-cz\sqrt{N/2}}$ by \eqref{eq:max_tail}. Hence, we get in total that
\[
\PP(\tau' = k|\,\tau' > k-1) \le e^{-cz\sqrt{N/2}},
\]
and $\tau'$ is dominated from below by a geometric random variable with success probability $e^{-cz\sqrt{N/2}}$. The proof now continues as above.
\end{proof}

\section{Proof of Lemma~\ref{lem:entropy}}
\label{sec:proof_lemma}

We first consider the case $\beta \in [0,\beta_c)$. Define
\[
D_{\beta,N} = \dv{}{\beta} W_{\beta,N} = \sum_{|u|=n} (X_u - \varphi'(\beta)N) e^{\beta X_u - \varphi(\beta)N}.
\]
We express the entropy by
\begin{align*}
H(\mu_{\beta,N}) &= \frac{1}{W_{\beta,N}} \sum_{|u|=N} (\varphi(\beta)N-\beta X_u)e^{\beta X_u - \varphi(\beta)N} + \log W_{\beta,N} \\
&= (\varphi(\beta)-\beta\varphi'(\beta))N - \beta \frac{D_{\beta,N}}{W_{\beta,N}} +\log W_{\beta,N}.
\end{align*}
By the assumption on $\beta$, we have $\varphi(\beta)-\beta\varphi'(\beta) > 0$. Furthermore, $W_{\beta,N}$ converges almost surely to a positive random variable as $N\to\infty$ by Fact~\ref{fact:W} and $D_{\beta,N}$ converges almost surely as well as $N\to\infty$, see \cite{Biggins92}. The first statement follows.

Now let $\beta > \beta_c$ and assume $\beta_c\in \mathcal D(\varphi)^\circ$. Define
\[
\widetilde W_{\beta,N} = \sum_{|u|=N} e^{\beta (X_u - \varphi'(\beta_c) N - \frac{3}{2\beta_c} \log N)}.
\]
We now write
\[
H(\mu_{\beta,N}) = \frac{1}{\widetilde W_{\beta,N}}\sum_{|u|=N} \left(-\beta \left(X_u - \varphi'(\beta_c)N - \frac{3}{2\beta_c} \log N\right)\right) e^{\beta (X_u - \varphi'(\beta_c)N - \frac{3}{2\beta_c} \log N)} + \log \widetilde W_{\beta,N}.
\]
Fix $\beta'\in(\beta_c,\beta)$. Then there exists $C>0$, such that $xe^{-\beta x} \le C e^{-\beta' x}$ for every $x\in\R$. We have
\[
H(\mu_{\beta,N}) \le C\frac{\widetilde W_{\beta',N}}{\widetilde W_{\beta,N}} + \log \widetilde W_{\beta,N}.
\]
Now, $\widetilde W_{\beta,N}$ and $\widetilde W_{\beta',N}$ converge almost surely as $N\to\infty$ to positive random variables, see \cite{Madaule17}. The second statement follows.

\section{Open questions and further directions} \label{sec:open questions}

In this section, we state a few questions and further directions related to the CREM and the branching random walks for future study.

\begin{enumerate}

\item It was conjectured in \cite{AB&M20} that for the CREM,
there exists a threshold $\beta_G\geq 0$ such that its Gibbs measure with parameter $\beta$ can be
efficiently approximated if $\beta < \beta_G$ and cannot if $\beta > \beta_G$. A conjectured explicit expression\footnote{There is a typo in the expression of $\beta_G$ in the case of the CREM in Item 1, Section 5 of \cite{AB&M20}. The definition of $t_G$ should be replaced by the following one: $t_G = \sup\{t\in[0,1]: A(s) = \hat A(s)\text{ for all }s\le t\}$.} of $\beta_G$ appears in  Item 1, Section 5 of \cite{AB&M20}. Our results confirm the conjecture in the case where the CREM has correlation function $A(x)=x$, with the notion of \emph{approximation} from Definition~\ref{def:approximation}. Moreover, our results imply that $\beta_G = \beta_c = \sqrt{2\log 2}$, the \emph{(static) critical inverse temperature}.
One can check that $\beta_c$ equals the expression of $\beta_G$ from \cite{AB&M20}. Ongoing work of the authors is trying to generalize the result to the CREM with a general correlation function.

\item Back to the branching random walk, one might be interested in the near critical regime to understand how the transition happens near $\beta_c$. To do this, one can take a sequence $\beta(N) = \beta_c-N^{-\delta}$ for some $\delta>0$. It might be interesting to study the time complexity of any algorithm approximating the Gibbs measure $\mu_{\beta(N),N}$ in the sense of Definition~\ref{def:approximation}. One should expect a phase transition at $\delta=1/2$, in line with a phase transition for the asymptotics of the partition function obtained by Alberts and Ortgiese \cite{albertsNearcriticalScalingWindow2013}. See also the introduction of Pain \cite{Pain18}. This should be related to Pemantle's \cite{PemantleSearch09} study of optimization algorithms discussed in Section~\ref{sec:related work}.
\end{enumerate}

\paragraph{Acknowledgements.} We are grateful to an anonymous referee for several helpful suggestions improving the presentation.




\providecommand{\bysame}{\leavevmode\hbox to3em{\hrulefill}\thinspace}
\providecommand{\MR}{\relax\ifhmode\unskip\space\fi MR }
\providecommand{\MRhref}[2]{%
  \href{http://www.ams.org/mathscinet-getitem?mr=#1}{#2}
}
\providecommand{\href}[2]{#2}

\end{document}